\documentclass[12pt,a4paper]{amsart}
\usepackage{amsfonts,amsmath,amssymb,amsthm}
\usepackage[width=16cm,height=23cm]{geometry}

\usepackage[utf8]{inputenc}
\usepackage[T1]{fontenc}

\usepackage{comment}
\usepackage{color}

\newtheorem{dfi}{Definition}[section]
\newtheorem{lem}[dfi]{Lemma}
\newtheorem{pro}[dfi]{Proposition}
\newtheorem{thm}[dfi]{Theorem}

\newtheorem{ex}[dfi]{Example}

\newcommand{\lqv}{{L^{q(\cdot)}(\Omega)}}
\newcommand{\lpvo}{{L^{p(\cdot)}[0,1]}}
\newcommand{\lqvo}{{L^{q(\cdot)}[0,1]}}

\newcommand{\esr}{{R_{p(\cdot)}}}
\newcommand{\rw}{\rightarrow}
\newcommand{\lpv}{{L^{p(\cdot)}(\Omega)}}
\newcommand{\mi}{{p_{|A_k}^-}}
\newcommand{\ma}{{p_{|A_k}^+}}

\begin{document}

\setlength{\baselineskip}{16pt}

\title[Variable exponent spaces. Structure.]{On the structure of variable  exponent spaces}
\author[J. Flores, F.L. Hern\'{a}ndez, C. Ruiz and  M. Sanchiz]{Julio Flores $^{1,2}$, Francisco L. Hern\'{a}ndez $^1$, C\'{e}sar Ruiz$^1$ and  Mauro Sanchiz}
\address{Departamento de Matemática Aplicada, Ciencia e Ingeniería  de Materiales y Tecnología Electrónica, ESCET, Universidad Rey Juan Carlos, 28933, Móstoles, Madrid, Spain.}
\email[Julio Flores]{julio.flores@urjc.es}
\address{Department of Mathematical Analysis and Applied Mathematics and imi, Faculty of Mathematics, Complutense University of Madrid, 28040 Madrid, Spain.}
\email[Francisco L. Hern\'{a}ndez]{pacoh@ucm.es}
\email[C\'{e}sar Ruiz]{cruizb@mat.ucm.es}
\email[Mauro Sanchiz]{msanchiz@ucm.es}
\thanks{$^1$Partially supported by grant MTM2016-76808-P}
\thanks{$^2$Partially supported by grant PGC2018-101625-B-I00}
\date{November 2019}
\dedicatory{Dedicated to the memory of Wim Luxemburg}
\begin{abstract}
The  first part of this paper  surveys several results on the lattice structure of variable exponent Lebesgue function spaces (or Nakano spaces) $\lpv$. In the second part strictly singular and disjointly strictly singular operators between spaces $\lpv$ are studied. New results  on  the  disjoint strict singularity of the  inclusions $ L^{p(\cdot)}(\Omega)  \hookrightarrow  L^{q(\cdot)}(\Omega)$  are given.

\end{abstract}
\maketitle

\section{Introduction}
Classical  variable exponent Lebesgue spaces (or  Nakano  spaces)\, $\lpv$ \, have recently been successfully used in some areas of Harmonic Analysis and PDEs.  (\cite{Libro},\cite{CF}). This has motivated in recent years a new interest in studying  the geometry and the structure of these spaces (\cite{P-R},\cite{LPP},\cite{HR1},\cite{AM}). Variable exponent Lebesgue spaces belong to the  general class of (non-rearrangement invariant) Musielak-Orlicz spaces (\cite{Musielak}).

This paper has two different parts. In the first part we collect, in a sort of survey, several recent  results on the lattice structure of variable exponent spaces. Our interest will be bounded to the existence of isomorphic copies and lattice isomorphic copies of sequence spaces $l_q$ ($1\leq q<\infty$) in spaces $\lpv$. Also, the existence of complemented copies  will be considered as well the subprojectivity  of these spaces 
(\cite{ HR1},\cite{HR2},\cite {R-S}). In general, the fact that $\lpv$ spaces are not rearrangement invariant  is a serious obstacle for obtaining bounded averaging projections. We remark that no extra conditions of regularity on the exponent functions, such as the local 
log-Hölder continuity condition (\cite{CF},\cite{Libro}),  will be required throughout the paper.

In the second part of the paper we study strictly singular and disjointly strictly singular operators between variable exponent spaces $ \lpv $. We obtain new results for the  disjoint strict singularity of the  inclusions \, $ L^{p(\cdot)}(\Omega)  \hookrightarrow  L^{q(\cdot)}(\Omega)$ by providing some  sufficient or necessary conditions (see \ref{Propo-1}, \ref{Prop-2} and \ref{lem-3}). Examples are also given showing that even in the case  \,$ess\inf (p-q) = 0$\,  the inclusion \, $ L^{p(\cdot)}(\Omega)  \hookrightarrow  L^{q(\cdot)}(\Omega)$ can be  disjointly strictly singular (see Example \ref{ejemp-2}).

Throughout the  paper, $(\Omega,\Sigma,\mu)$ is
a $\sigma$-finite  separable non-atomic measurable space and
$L_{0} (\Omega)$ is the space of all real measurable  function
classes. Given  a $\mu$-measurable function $p : \Omega
\rightarrow [1,\infty) $,
 the {\it Variable Exponent Lebesgue space }(or Nakano  space),  $\lpv$,
 is defined as the set of all measurable scalar function
classes $f \in L_0(\Omega)$  such that the  modular
$\rho_{p(\cdot)}(f/r)$ is finite for some $r>0$ ,  where
$$\rho_{p(\cdot)}(f)\,= \,\int _{\Omega} |f(t)|^{p(t)}\, d\mu(t) < \infty.$$
The associated Luxemburg norm is defined as
 $$ ||f||_{p(\cdot)} := \inf \{ r>0 ; \ \rho_{p(\cdot)}(f/r)\leq 1 \}
 $$

 With the usual pointwise order, $\big(\lpv , ||\,||_{p(\cdot)}\big)$
 is a Banach lattice.
We write
$ p^- := \text{ess} \inf
    \{ p(t): t\in \Omega\}$ and 
    $p^+ := \text{ess} \sup\{
    p(t): t\in
    \Omega \}.$
 Equally, \,$p_{|B}^+$\, and \,$p_{|B}^-$\, will denote the essential
    supremum and infimum of the function \,$p(\cdot)$\, over  a measurable
  subset $B$ of $\Omega$. The conjugate  function $p^*(\cdot)$  of  $p(\cdot)$ is defined by  the equation
    $\displaystyle \frac{1}{p(t)}+\frac{1}{p^*(t)}= 1$ almost everywhere. 
    
        As it should be expected, the properties of variable exponent spaces depend on  the measurable exponent functions $p(\cdot).$
     Thus,
the topological dual  of the space $\lpv$, for $p^+<\infty$,  is the
    variable  exponent space   \,$L^{p^*(\cdot)}(\Omega)$.
Also,  $\lpv$ is separable if and only if  $p^+< \infty$ or, equivalently, if and only if $\lpv$ contains no isomorphic copy of $l_{\infty}$.
In the sequel only separable exponent Lebesgue spaces $\lpv$ will be considered. In this case ($p^+< \infty$), notice  that $||f||_{p(\cdot)}=1$ if and only if $\rho_{p(\cdot)}(f)= 1$. Also, a sequence $(f_n) \subset \lpv$ satisfies $\lim_{n\rw \infty}||f_n||_{p(\cdot)}=0$  if and only if $\lim_{n\rw \infty}\rho_{p(\cdot)}(f_n)\,= 0$.  
The space $\lpv$ is reflexive if and only if  \,$1<p^{-}\leq
    p^{+}<\infty$ (this is also equivalent to   being uniformly convex). Also,
    $\lpv$  is a  \,${p^-}$-convex and \,$p^+$-concave Banach lattice.
    
    The  \textit{essential range} of the exponent function
$p(\cdot)$ is
    defined as
    $$ \esr := \{q\in[1,\infty) :\quad \forall \epsilon >0
    \quad \mu( p^{-1}(q-\epsilon,q+\epsilon))>0\}.
    $$
Note that the essential range is a closed subset of
$[1,\infty)$; in particular,  $\esr$  is  compact  when
$p(\cdot)$ is essentially bounded. Clearly, the values  \,$p^{-}$
and $p^{+}$ are always in the set \,$ \esr $.

We refer the reader to (\cite{CF},\cite{Libro}) for standard definitions and properties regarding variable Lebesgue spaces. Also, we refer  to \cite{AK}, \cite{LT1}, \cite{LT2}, \cite{LuxZa}, \cite{Me} and \cite{Za} for the standard  Banach lattice and Banach space  terminology. Particularly, a  Banach lattice $X$ will be said to have an isomorphic  \em copy \em of the Banach space $Y$ if there is  an isomorphism $\varphi:Z\rightarrow Y$ where $Z$  a subspace of $X$.  If additionally $Y$ is a Banach lattice, $Z$ is a sublattice of $X$ and $\varphi$  preserves the lattice structure (Riesz isomorphism), then we will say that $X$ has a \em lattice isomorphic copy \em of $Y$.

\section{$\ell_{q}$ -structure of variable exponent spaces}\label{estructura}

As one can expect, the structure of variable exponent spaces  is richer than that of  $L^p(\Omega)$ spaces.
In this section we will focus on the existence of  both isomorphic copies and lattice-isomorphic  copies of sequence spaces $l_q$ $(1\leq q <\infty)$  in spaces $\lpv$. Also, some results regarding projections and complementation will be given.

With respect to the existence of lattice isomorphic copies we remind the reader that in  a classical $L ^p(\Omega)$ space a lattice-isomorphic  copy of $l^q$ is well-known to exist only when $p=q$. For  Nakano spaces the situation is more involved.
Evidently,  if $\mu(p^{-1}(\{q\}))>0$ then $\lpv$ contains a lattice isomorphic copy of $L^{q}(\Omega)$ , and hence of  $l^q$. In the general case we need to look at the essential range  $R_{p(\cdot)}$ and proceed as follows.
For every $q\in \esr$, choose a sequence \,$(A_k)$\, of
disjoint measurable sets with $0<\mu(A_k)<\infty$ such that \,\,$A_{k} \subseteq
p^{-1}(\,[q+\frac{1}{k+1},q+\frac{1}{k})\, )$\,
 for every  \,$k$  (or  \,\,$A_{k} \subseteq
p^{-1}(\,[q-\frac{1}{k},q-\frac{1}{k+1})\, )$). Consider next
the sequence of Orlicz functions \,$(\psi_k)$\, defined as
        $$
\psi_{k}(s):= \frac{1}{\mu(A_k)} \int_{A_k} s^{p(t)}
d\mu(t)\,\,\;\quad \quad
        $$
for $s\in[0,1]$\, and  $ k\in\mathbb{N}$. Since \,\,$s^{q+\frac{1}{k}} \leq \psi_k(s) \leq s^{q+\frac{1}{k+1}} \leq
s^q$ for every \,$0\leq s \leq 1$,\, it turns out that, up to passing to some  subsequence,  the Musielak-Orlicz sequence spaces  \,$l_{(\psi_{n})}$  must coincide with $l_{q}$. Finally,  the unconditional basic sequence $(g_k) \subset \lpv$
 $$ g_k  := \frac{\chi_{A_k}(t)}{\mu(A_k)^{\frac{1}{p(t)}}}.
 $$ is shown to be equivalent to the canonical basis of $l_q.$ This was observed in 
 \begin{pro}[\cite{HR1}]
\label{Resultado-1}
For every  $q\in \esr$ the space $\lpv$ contains a
lattice-isomorphic  copy  of  $l_{q}$.
\end{pro}
On the other hand, if  \,$(g_{k})$\, is a sequence of disjoint normalized functions in
the separable space \,$\lpv$,\,by the density of simple functions in $\lpv,$ we can consider $(g_k)$ as disjoint simple functions
$$
g_k = \sum_{j=1}^{N_k} a_{k,j}\,\chi_{A_{k,j}}.\,
$$  Now  consider the  Orlicz function sequence
        \,$(\psi_k)$\,
        defined by
        $$
        \psi_k(s): =\frac{1}{\rho_{p(\cdot)}(g_k)}\sum_{j=1}^{N_k}\int_{A_{k,j}}s^{p(t)}|a_{k,j}|^{p(t)}d\mu(t),
        \quad
        $$
        for $ s\in[0,1]$. On can prove that  the basic sequence \,$(g_k)$ is equivalent to  the canonical
basis   of the  Musielak-Orlicz sequence space  $l_{(\psi_k)}$. Therefore $l_{(\psi_k)}$ and
the closed span $[g_k]$ are lattice isomorphic. Moreover, since
$$p^- \ \leq \  \frac{s\,\psi_n '(s)}{\psi_n(s)}\ \leq \ p^+,
$$ the following is derived:

\begin{thm}[\cite{HR1}]
\label{Resultado-2} Let $\lpv$\, be   separable. Then, $\lpv$
 has a lattice-isomorphic  copy of \,$l_q$\, if and only if \,$q\in \esr$.
\end{thm}

The existence of isomorphic copies of $l_q$ in $\lpv$ was also considered in \cite{HR1} where the  proof  involves  the generalized Kadec-Pe\l
czynski method (\cite{LT2}, Prop. 1.c.8)

\begin{thm}[\cite{HR1}] \label{Resultado-3}
 Let $\lpv$ be with \,$p^+ < \infty$.

(a) If \,\,$1 \leq p^- \leq 2$, then $\lpv$ contains an isomorphic  copy of $l_q$ if and only if \,$q \in \esr \cup
[p^- ,2]$\,.

(b) If \,\,$p^- > 2$, then $\lpv$ contains an isomorphic copy of  \,$l_q$\,  if and only if \,\,$q \in \esr \cup \{2\}$.

\end{thm}

We now turn to the existence of $complemented$ isomorphic
copies of \,$l_q$\, in a $\lpv$. A  duality argument
allows us to consider  only scalars   \, $q$\,  in  the essential
range  set \,$ \esr$. As seen above,
   a lattice-isomorphic  copy of \,$l_q$ in \,
 $\lpv$,\, for \, $q\in \esr$,\, can be
obtained by considering the span of a  suitable sequence of disjoint
normalized functions of the form
\,$\displaystyle
\left(\frac{\chi_{A_k}(t)}{\mu(A_k)^{\frac{1}{p(t)}}}\right)$.
In order to have this copy complemented,  we might consider the associated  {\it orthogonal
} projections, $T_A$, defined as

 $$ T_A(f)(t)\,= \,\sum_{k=1}^{\infty} \left(\int_{A_k}
    \,\frac{f(s)}{\mu(A_k)^{\frac{1}{p^*(s)}}}d\mu(s)\right)\,\frac{\chi_{A_k}(t)}{\mu(A_k)^{\frac{1}{p(t)}}} ,
\vspace{2mm}$$  where \,$\frac{1}{p(t)}+ \frac{1}{p^*(t)}=1$\,
almost everywhere. Unfortunately these projections need not  be  bounded in general:

\begin{ex}\label{Resultado-4} $($\cite{HR2}$)$ Let $\lpv$ be  separable. If the
 essential range  $\esr$ is not just one single point, then there exist
  sequences of disjoint measurable sets $A=(A_k)$  such that the associated orthogonal projection $T_A$ is  not
 bounded.
\end{ex} Indeed,  consider $q_1<q_2$   elements in $\esr$; take $0< \delta<\frac{q_2 - q_1}{4}$
and  two measurable subsets  $R$ and  $S$ of $\Omega$ verifying

    $$ \begin{array}{ccc}R & \subset & p^{-1}(q_2 -\delta, q_2 +
    \delta)\\
    S & \subset & p^{-1}(q_1 -\delta, q_1 +
    \delta) \\
\end{array}
$$
 with $\mu(R)=\mu(S)>0$. Since $\mu$ is  non-atomic, it
is possible to split
 $R$ and $S$ as an union of mutually disjoint measurable subsets $\bigcup_k R_k = R$  and $\bigcup_k S_k = S$ with
 $$\displaystyle
\mu(R_k)=
\left(\frac{\mu(R)}{\sum_{n=1}^{\infty}\frac{1}{n^{\alpha}}}\right)\frac{1}{k^{\alpha}}=\mu(S_k),$$
for certain suitable  $\alpha > 1$ . Set $A_k= R_k
\bigcup S_k$\, for  $k\in \mathbb{N}$. Then, it can be proved  that the operator $T_A$ is no bounded.
Moreover, the associated averaging projection $P_A,$ defined by
$$P_{A}(f)\,:=\,\sum_{k=1}^{\infty}\frac{ \int_{A_k}
    \,f(s)d\mu(s)}{\mu(A_k)}\,\,\chi_{A_k},
$$ is neither  bounded.

    The following notion of a $p(\cdot)$-regular sequence is useful:

   \begin{dfi} \label{Resultado-5}  Let $(A_k)$ be a sequence
of disjoint measurable sets and  $p : \Omega
\rightarrow [1,\infty)$\, be a measurable function with
$p^+<\infty$. The sequence  $(A_k)$ is said to be
$p(\cdot)$\textit{-regular} if the associated  Nakano
sequence
 spaces $l^{(\ma)}$ and $l^{(\mi)}$ coincide and there exists a constant $C>0$ such that for every natural $k$,
 $$ \frac{1}{C} < \mu(A_k)^{\ma-\mi} < C.
 $$
   \end{dfi}

We remark (\cite{HR2}) that for a  $p(\cdot)$-regular sequence $(A_k)$  the associated orthogonal projection  $T_{A}$ and the averaging projection $P_{A}$ are bounded.
\vspace{2mm}

Notice that if the exponent functions $p(\cdot)$ are defined  on a open set $\Omega$ in $\mathbb{R}^{n}$ and satisfy the log-Hölder continuous condition (\cite{CF},\cite{Libro} Def. 4.1), then every sequence of disjoint balls  $(A_k)$ satisfying $\sum \mu(A_k)< 1$ must be $p(\cdot)$-regular  and hence the associated projections $P_{A}$ and $T_{A}$ are bounded (\cite{HR2}).

We do not know any characterization for  an orthogonal operator $T_A$  (or an averaging projection $P_A$) to be bounded  in  $\lpv$ spaces.

   \begin{thm}\label{Resultado-6} $($\cite{HR1}$)$ Every variable exponent space \,$\lpv$\, has a lattice-isomorphic
complemented  copy  of $l_q$ for every   $q\in \esr$.
\end{thm}

A proof of the above result can be given taking  suitable  $p(\cdot)$-regular   measurable subset sequences.

 Recall that a Banach space $X$ is {\em subprojective} if every infinite-dimensional subspace of $X$ has some infinite-dimensional subspace which is complemented in $X$. Notice that this notion is an isomorphic property. It is well known that classical  $L^p(\Omega)$ spaces are subprojective if and only if $2 \le  p <\infty$ (\cite{Wh}).
Coherently with the  situation for $L^p(\Omega)$ spaces,  we have the  following extension:

\begin{thm} \label{Resultado-7} $($\cite{R-S}$)$
Let $L^{p(\cdot)}(\Omega)$\, be a separable variable exponent  space with $\mu(\Omega)<\infty$. If $p^-\ge2$, then \,$L^{p(\cdot)}(\Omega)$\,  is subprojective.
\end{thm} In view of Theorems \ref{Resultado-2} and \ref{Resultado-3},  the  infinite-dimensional subspace of $\lpv$ which appears in the proof of the Theorem must be  an $l_q$-space with $q\in \esr$ or $q=2$.
In addition, by Theorem \ref{Resultado-3}, the condition $2\le p^-$ turns out to be also necessary (see   5.4 in \cite{R-S}).

Recall also that a Banach space $X$ is $superprojective$ if  every infinite-codimensional subspace of $X$ is contained in an infinite-codimensional subspace complemented in $X$. As a  consequence of the  above result and  the reflexivity,  we have that \,  $L^{p(\cdot)}(\Omega)$  is superprojective if and only if  $ 1 < p^{-} \leq p^{+} \leq 2 $.

An open problem is obtaining  conditions on the exponent function $p(\cdot)$ so that  $\lpv$  contains a subspace isomorphic to $L^{q}(\Omega)$ for $2<q<\infty$.  Of course, the trivial case   $\mu(p^{-1}(\{q\}))>0 $ can be excluded here.

\section{Strictly and disjointly strictly singular operators }\label{operadores}

In this section we study strictly singular and disjointly strictly singular operators between variable exponent spaces, presenting  new results for inclusion operators. 

Let X and Y be two Banach spaces. Recall that a bounded lineal  operator $T : X \rw Y$ is  \em strictly singular \em (or \em Kato\em)   if for every infinite-dimensional closed subspace $Z$ of $X,$ the restriction $T_{|Z}$ is not an isomorphism. If, additionally, $X$ is a Banach lattice the operator $T$ is said to be \textit{disjoint strictly singular }($DSS$)  if for every pairwise disjoint normalized  sequence $(f_n)\subset X,$ the restriction to the span  $T_{|[f_n]}$ is not an isomorphism. In general, we can also talk about an operator being
 \em $Z$-strictly singular \em if for every  closed subspace $M\subseteq X$ isomorphic to the Banach space $Z$,  the restriction $T_{|M}$ is not an isomorphism. For basic properties of these   classes of operators  we refer  to (\cite{AK},\cite{LT1}).

   In the case of  classical $L^{p}(\Omega) $ spaces, the following characterization is well known (\cite{W}): An operator $T: L^{p}(\Omega)\rightarrow L^{p}(\Omega)$ is strictly singular if and only if it is $\ell^{p}$-strictly singular and $\ell^{2}$-strictly singular.

  In order to  present extensions of this result in the context of  variable exponent spaces, let us recall that a subspace $X$ of a $\lpv$ is \em strongly embedded \em  if convergence in the $\lpv$-norm in $X$ is equivalent to the convergence in measure. It turns out  that a subspace $X$ is strongly embedded in $\lpv$ if and only if there exist $\delta>0$ and another variable exponent space $L^{q(\cdot)}$ with  $1\leq q(\cdot)+\delta \leq  p(\cdot)$ such that  the norms \, $\|\,\|_{p(\cdot)}$\,  and \,  $\|\,\|_{q(\cdot)} $
   are equivalent on $X$. In this context,  an operator $T:\lpv\rightarrow \lpv $ which is invertible on a strongly embedded subspace cannot be  $\ell^{2}$-strictly singular. Using this fact and the above characterization of complemented $l_q$-copies in  $\lpv$  we have the following:

\begin{thm} \label{Resultado-8} $($\cite{BH}$)$ Let $T : \lpv \rw \lpv $ be a bounded operator,  $\mu(\Omega ) <\infty$, and $  1 \leq  p^-\leq p^+ <\infty.$ The following statements are equivalent:

i)  $T$ is strictly singular.

 ii) $T$ is $l_q$-strictly singular for every $q \in \esr \cup \{2\}.$
 
iii)  $T$ is DSS and $l_2$-strictly singular.
\end{thm}

Notice that the equivalence between  (i) and (iii) remains true for  a general Banach lattice $E$ satisfying a lower $2$-estimate (\cite{FHKT}).

We pass now to consider in further detail DSS   operators  on variable exponent  spaces. More precisely, we will focus on the disjoint strict singularity  of inclusions between different $\lpvo$ spaces looking for suitable criteria. In what follows, we will restrict ourselves to   $\Omega = [0,1]$ with the usual Lebesgue measure and bounded exponent (measurable) functions $ p, q  :  [0,1] \rw  [1,M] $ satisfying $q(x)\le p(x)\le p^+\le M<\infty$. Notice that, in this context, the inclusion
$ I : \lpvo \hookrightarrow \lqvo
$ is bounded.

We make the observation that the condition  $\esr \cap R_{q(\cdot)}=\emptyset$  implies that the inclusion  $I : \lpvo \hookrightarrow \lqvo $ is DSS. This condition on the essential ranges is evidently too strong and, in  particular,  implies that  $ess\inf (p-q) >0$,  which  turns out to be also a sufficient condition.
The proof is straigthforward and makes use of the following standard lemma:
 \begin{lem} \label{lem-1} 
 Let
 $(f_n)$ be a pairwise disjoint normalized sequence in  $\lpvo$ and   $(r_n)\uparrow \infty$  a sequence of scalars. Then,  up to passing to subsequence or equivalence, we can assume that $ |f_n(x)| > r_n$ for all $x\in \text{supp}f_n$.
\end{lem}
     \begin{proof} Take  $A_n= \text{supp}f_n$ pairwise disjoint. Evidently,
     $ \sum \mu(A_n) \leq 1.
     $  Define
     $ B_{n,1}= \{x  :  |f_n(x)|> r_1\}$ and  $ C_{n,1}= [0,1]\backslash B_{n,1}$. %\{x  :  |f_n(x)|\leq r_1\}.$
      Notice that $\mu(B_{n,1})>0$ as $(f_n)$ is normalized. Let $g_n = f_n \chi_{B_{n,1}}$. Since $1\leq p^+ < \infty$, we have
     $$ \rho_{p(\cdot)} (f_n-g_n)= \int_{C_{n,1}}|f(x)|^{p(x)}d\mu(x) \leq \mu(C_{n,1})r_1^{p^+} \rw 0,
     $$ or equivalently
     $ ||f_n - g_n||_{p(\cdot)}\rw 0$. From here, a standard perturbation argument (\cite{LT1} Prop. 1.a.9) concludes that some subsequence $(f_{n_k})$ can be assumed to meet our claim.
     \end{proof}
     
     \begin{pro} \label{Propo-1} If  $ess\inf (p-q) >0$, then the inclusion  $I : \lpvo \hookrightarrow \lqvo $ is DSS.
     \end{pro}

     \begin{proof} Let $(f_n)\subset \lpvo$ be a pairwise disjoint normalized sequence.
     Assume, as in  Lemma \ref{lem-1}, that $|f_n|> n,$   and \,$||f_n||_{p(\cdot)} = 1$ . Choosing $0< \delta< ess\inf (p-q)$, we have
     \begin{align*}
     \rho_{q(\cdot)}(f_n)&= \int_{A_n} |f_n(x)|^{q(x)} dx=\int_{A_n} |f_n(x)|^{p(x)} |f_n(x)|^{q(x) - p(x)}dx\\
     &=
    \int_{A_n} |f_n(x)|^{p(x)} (\frac{1}{|f_n(x)|})^{p(x) - q(x)}dx \leq (\frac{1}{n})^{\delta}\int_{A_n} |f_n(x)|^{p(x)} dx\\
    &=(\frac{1}{n})^{\delta}\rw 0,
     \end{align*}

     and thus $||f_n||_{q(\cdot)}\rw 0.$ Therefore,  $I_{|_{[f_n]}}$  cannot be  an isomorphism.
     \end{proof}
%%%%%%%%%%%%%%%%5
The following result shows that condition $ess\inf (p-q) >0$ above can be replaced by  a weaker assumption:

\begin{pro} \label{Prop-2} Let $p(\cdot)$ and $q(\cdot)$ be bounded exponent functions  with $q(\cdot)\le p(\cdot)\le p^+<\infty$ and such that $p-q= f(\cdot)$ is decreasing. Assume that there is an increasing sequence $(x_n) \uparrow 1$ such that
$$ \sum_n (x_{n+1} -x_n)^{\frac{f(x_{n+1})}{p^+}} <\infty.
$$ Then, the inclusion $I : \lpvo \hookrightarrow \lqvo $ is DSS.
\end{pro}

We remark that that the  assumption in Proposition \ref{Prop-2} is clearly implied by condition $ess\inf (p-q)>0$  and yet compatible with $R_{p(\cdot)}\cap R_{q(\cdot)}\neq\emptyset$. 
This is shown with the following example:

\begin{ex} \label{ejemp-2} Consider a bounded exponent function $q(\cdot)$ and define \, $p(\cdot) = q(\cdot) + f(\cdot)$, where $f$ is the decreasing function
 $$f =\chi_{\left[0,\frac{1}{2}\right)} + \sum_{n\geq 2} \,\frac{1}{\ln n} \,\, \chi_{\left[1-(\frac{1}{n})^{n-1},1-(\frac{1}{n+1})^{n}\right)}$$ \,
Then, $ess\inf f(x)=0$ and the sequence  \,$x_n=1-(\frac{1}{n})^{n-1}\nearrow 1$ \, satisfies
  $$\sum_n (x_{n+1} -x_n)^{\frac{f(x_{n+1})}{1+q^{+}}} \,< \,\sum_n (\frac{1}{n})^{\frac{n-1}{(1+q^{+})\ln (n+1)}} <\infty
$$
\end{ex}

\begin{proof}[Proof of Prop. \ref{Prop-2}] Let $(f_k) \subset \lpvo$ be a normalized pairwise disjoint sequence and $\epsilon > 0.$  By the assumption, there is $n_0$ such that
$$ \sum_{n\geq n_0} \,(x_{n+1} -x_n)^{\frac{f(x_{n+1})}{p^+}} \leq \frac{\epsilon}{3}. \qquad (*)
$$

With the following notation:
$$ A_k= supp f_k,\quad A_{k,0} = A_k \bigcap [0,x_{n_0}),\quad
 A_{k,n} = A_k \bigcap [x_n, x_{n+1}), \  n\geq n_0;
$$
$$ B_{k,n}=\{t\in A_{k,n}\,:\, |f_k(t)| \leq \frac{1}{(x_{n+1}-x_n)^{\frac{1}{p(t)}}}\, \} \quad  n\geq n_0;
$$
$$ C_{k,n}=\{t\in A_{k,n}\,:\, |f_k(t)| > \frac{1}{(x_{n+1}-x_n)^{\frac{1}{p(t)}}}\, \} \quad n\geq n_0;
$$
 we can write

$$ \rho_{q(\cdot)}(f_k) = \int_{A_{k,0}} |f_k(t)|^{q(t)}dt + \sum_{n\geq n_0}\int_{B_{k,n}} |f_k(t)|^{q(t)}dt + \sum_{n\geq n_0}\int_{C_{k,n}} |f_k(t)|^{q(t)}dt.
$$

Since $f(t) \geq f(x_{n_0})>0$ for all $t<x_{n_0}$, it follows from Proposition \ref{Propo-1} %and Remark \ref{Observac-1}  
that, up to some subsequence, there exists  $k_0$ such that
$$\rho_{q(\cdot)} (f_k\chi_{A_{k,0}}) = \int_{A_{k,0}} |f_k(t)|^{q(t)}dt <\frac{\epsilon}{3}$$
for all $k\geq k_0$.
On the other side, since $f=p-q$ is decreasing, we have
\begin{align*} \sum_{n\geq n_0}\int_{B_{k,n}} |f_k(t)|^{q(t)}dt &\leq \sum_{n\geq n_0} \int_{x_n}^{x_{n+1}} \frac{1}{(x_{n+1}-x_n)^{\frac{q(t)}{p(t)}}} dt =\sum_{n\geq n_0} \frac{1}{x_{n+1}-x_n} \int_{x_n}^{x_{n+1}} \frac{1}{(x_{n+1}-x_n)^{\frac{q(t)}{p(t)}-1}} dt \\
 &\leq \sum_{n\geq n_0} \frac{1}{x_{n+1}-x_n} \int_{x_n}^{x_{n+1}} (x_{n+1}-x_n)^{\frac{f(x_{n+1})}{p^+}} dt = \sum_{n\geq n_0}  (x_{n+1}-x_n)^{\frac{f(x_{n+1})}{p^+}} \leq \frac{\epsilon}{3},
\end{align*}
 where  $(*)$ has been used. Evidently, $(*)$ also implies  $(x_{n+1} -x_n)^{\frac{f(x_{n+1})}{p^+}} \leq \frac{\epsilon}{3}$. Since  $f$ is decreasing and $(f_k)$ is normalized,  we have
 \begin{align*}
 \sum_{n\geq n_0}\int_{C_{k,n}} |f_k(t)|^{q(t)}dt &\leq \sum_{n\geq n_0} \int_{C_{k,n}} |f_k(t)|^{p(t)}|f_k(t)|^{q(t)-p(t)}dt \\
 &\leq
\sum_{n\geq n_0} \int_{x_n}^{x_{n+1}} |f_k(t)|^{p(t)} (x_{n+1}-x_n)^{\frac{p(t)-q(t)}{p(t)}} dt \\
&\leq
  \sum_{n\geq n_0} \int_{x_n}^{x_{n+1}} |f_k(t)|^{p(t)} (x_{n+1}-x_n)^{\frac{f(x_{n+1})}{p^+}} dt \\
  &   \leq \sum_{n\geq n_0} \int_{x_n}^{x_{n+1}} |f_k(t)|^{p(t)} \frac{\epsilon}{3} dt \leq \frac{\epsilon}{3}.
\end{align*}

Thus, $\rho_{q(\cdot)}(f_{k}) \leq \epsilon$ for every  $ k \geq k_{0}$
\end{proof}

Considering for instance $ q(\cdot)=\chi_{[0,1]}$ and \, $p(\cdot) = q(\cdot) + f(\cdot)$, with  $f(\cdot)$ as above, it is evident that $ess\inf (p-q) = 0$ and yet $ I:  L^{p(\cdot)}[0,1] \hookrightarrow L^{1}[0,1]$ is DSS.

Reasoning in a similar way, we have the following:

\begin{pro} Let $p(\cdot)$ and $q(\cdot)$ be bounded exponent functions  with $q(\cdot)\le  p(\cdot)\le p^+<\infty$. Assume that there exists a sequence $(r_k)\searrow 0$ such that the measurable sets $R_k:= (p-q)^{-1}(\,(r_{k+1},r_k]\,)$ satisfy
$$  \sum_k  \mu(R_k) ^{\frac{(p-q)_{|R_k}^-}{p^+}}< \infty.
$$ Then
the inclusion
$  I : \lpvo \hookrightarrow \lqvo
$ is DSS.
\end{pro}

When it comes to obtaining necessary conditions we have the following result:

 \begin{pro} \label{lem-3} Let $p(\cdot)$ and $q(\cdot)$ be bounded exponent functions  with $q(\cdot)\le p(\cdot)\le p^+<\infty$   such that $p-q$ is decreasing. If
 the inclusion \,$I : \lpvo \hookrightarrow \lqvo $\,  is DSS,  then
     $$ \lim_{x\rw 1} (1-x)^{(p-q)(x)} =   0.$$

     \end{pro}

     \begin{proof} Otherwise,
     there exist $r>0$ and some sequence  $(x_n)\nearrow 1$  such that
     $$ (1-x_n)^{(p-q)(x_n)}\ge r> 0
     $$ for all $n$. In fact, there is no loss of  generality in assuming $\frac{x_n +1}{2} < x_{n+1} $ and consequently  $(p-q)(\frac{x_n+1}{2})> (p-q) (x_{n+1})$ for all $n$. Call
     $$ A_n= (p-q)^{-1} ([(p-q)(\frac{x_n+1}{2}),(p-q)(x_n)])\supseteqq[x_n, \frac{x_n+1}{2}].
     $$ Notice that $(A_n)$ is pairwise disjoint as  $\frac{x_n +1}{2} < x_{n+1} $; in addition  $\mu(A_n)\geq \frac{1-x_n}{2}.$ Consider the sequence 
     $$ s_n = \frac{\chi_{A_n}}{\mu(A_n)^{\frac{1}{p(x)}}},
     $$ which is  pairwise disjoint and normalized in $\lpvo.$ Let $[s_n]$ be its closed span. We will prove that the restriction $I|_{[s_n]}$ is an isomorphism.
     For this, it suffices to check that $(s_n)$ and $(Is_n)$ are equivalent basic sequences. Since  $I$ is bounded, we only need to prove that 
     $\sum y_ns_n \in \lqvo $ implies $\sum y_ns_n \in \lpvo.$

     Observe that
     $$\rho_{p(\cdot)}( \sum y_n s_n) = \sum \int_{A_n} |y_n|^{p(x)}\frac{\chi_{A_n}}{\mu(A_n)}dx=
     \sum \int_{A_n} |y_n|^{q(x)}|y_n|^{(p-q)(x)}\frac{\chi_{A_n}}{\mu(A_n)^{\frac{q(x)}{p(x)}}\mu(A_n)^{1-\frac{q(x)}{p(x)}}}dx.
     $$ Notice that  $|y_n|<1$ perhaps up to  finitely many terms. Indeed, since $\mu(A_n)<1$ and  $\frac{1}{M}\leq\frac{1}{p(x)}\leq 1$ we would otherwise have
     \begin{align*} \rho_{q(\cdot)} (\sum y_ns_n) &\geq \sum \int_{A_n} \frac{1}{ \mu(A_n)^{\frac{q(x)}{p(x)}}}dx =
          \sum \int_{A_n} \frac{1}{\mu(A_n)^{\frac{p(x)+q(x)-p(x)}{p(x)}}}dx\\
          &=\sum \int_{A_n} \frac{1}{\mu(A_n)^{1-\frac{p(x)-q(x)}{p(x)}}}dx =
      \sum \int_{A_n} \frac{1}{\mu(A_n)} \mu(A_n)^{\frac{p(x)-q(x)}{p(x)}}dx \\
      &\geq \sum \int_{A_n} \frac{1}{\mu(A_n)} \mu(A_n)^{(p(x)-q(x))^+_{|_{A_n}}}dx \geq \sum \int_{A_n}\frac{1}{\mu(A_n)} (\frac{1-x_n}{2})^{(p-q)(x_n)}dx\\
      &\geq
          \sum r\int_{A_n}\frac{1}{\mu(A_n)} dx =\infty,
     \end{align*}
     which is not possible. From here we can write

 \begin{align*}
  \rho_{p(\cdot)} (\sum y_ns_n)&=\sum \int_{A_n} |y_n|^{q(x)}|y_n|^{(p-q)(x)}\frac{\chi_{A_n}}{\mu(A_n)^{\frac{q(x)}{p(x)}}\mu(A_n)^{1-\frac{q(x)}{p(x)}}}dx\\
    &\leq
    \sum \int_{A_n} |y_n|^{q(x)}\frac{\chi_{A_n}}{\mu(A_n)^{\frac{q(x)}{p(x)}}}[\frac{1}{\mu(A_n)}]^{1-\frac{q(x)}{p(x)}}dx \\
     &\leq
     \sum \int_{A_n} |y_n|^{q(x)}\frac{\chi_{A_n}}{\mu(A_n)^{\frac{q(x)}{p(x)}}}[\frac{1}{\mu(A_n)}]^{(p(x)-q(x))^{+}_{|_{A_n}}}dx\\
  &\leq \sum [\frac{1}{(\frac{1-x_n}{2})}]^{(p-q)(x_n)} \int_{A_n}|y_n|^{q(x)}\frac{\chi_{A_n}}{\mu(A_n)^{\frac{q(x)}{p(x)}}}dx \leq \frac{2^{p^+}}{r}\,\rho_{q(\cdot)}(\sum y_ns_n) <\, \infty.
  \end{align*}
  and the proof is finished.
     \end{proof}

As an application, consider the family of functions $$p(x)= q+(1-x)^{r}$$ on $[0,1]$ where $q\in[1,\infty)$ is fixed and $r$ runs in $[1,\infty)$. From the above follows that the  inclusions $ L^{p(\cdot)}[0,1] \hookrightarrow L^{q}[0,1]$ are  not DSS and, consequently (by \cite{GHR} Prop. 2.3), no $q$-convex rearrangement invariant space  $E[0,1] $ can be found between $L^{p(\cdot)}[0,1]$ and  $L^{q}[0,1]$.

We remark that in the previous proposition the assumption on the function $p-q$ being decreasing can be circumvented by considering $(p-q)^*$, the decreasing rearrangement of $p-q$. Indeed,  notice that
$ (p-q)^* : [0,1] \rw [0,M]$ and also $ \lim_{x\rw 1} (1-x)^{(p-q)^*(x)} \neq 0$.
 Defining as above
$$A_n = (p-q)^{-1} ([(p-q)^*(\frac{x_n+1}{2}),(p-q)^*(x_n)])$$ which also satisfies $ \mu(A_n)\geq\frac{1-x_n}{2}$, the following result is similarly obtained:

\begin{pro} \label{lem-3bis}  Let $p(\cdot)$ and $q(\cdot)$ be bounded exponent functions  with $q(\cdot)\le p(\cdot)\le p^+<\infty$. If
 the inclusion $I:\lpvo \hookrightarrow \lqvo $ is DSS,  then
     $$ \lim_{x\rw 1} (1-x)^{(p-q)^*(x)} =  0.$$

\end{pro}

Despite Propositions \ref{Prop-2} and \ref{lem-3} provide sufficient and necessary conditions for the   inclusion $I: \lpvo \hookrightarrow \lqv $  to be $DSS$, a complete characterization  still remains unknown.

\end{document}